\documentclass[12pt, reqno]{amsart}
\usepackage{amsmath,amssymb,amsthm}
\usepackage{mathtools}
\usepackage{stmaryrd} 
\usepackage[inline]{enumitem}
\usepackage{graphicx,setspace}
\usepackage{tikz}
\usetikzlibrary{arrows}
\usetikzlibrary{intersections, decorations.markings}
\usepackage{wasysym}
\usepackage[makeroom]{cancel}
\usepackage{fullpage}
\usepackage{comment}
\usepackage{caption}
\usepackage{mathtools}
\usepackage{url}
\usepackage{multirow}
\usepackage{soul}
\usepackage{chngcntr}
\usepackage{mathrsfs}
\usepackage{amsthm}
\usepackage{enumitem}

\newlist{renumerate}{enumerate}{3}
\setlist[renumerate]{label=\roman*),before=\raggedright}
\newlist{arenumerate}{enumerate}{3}
\setlist[arenumerate]{label=\arabic*),before=\raggedright}
\newlist{aenumerate}{enumerate}{3}
\setlist[aenumerate]{label=\alph*),before=\raggedright}

\counterwithin{equation}{section}
\allowdisplaybreaks

\newcommand{\znz}[1]{\mathbb{Z}/#1\mathbb{Z}}
\newcommand{\znzznz}[2]{\mathbb{Z}/#1\mathbb{Z} \times \mathbb{Z}/#2\mathbb{Z}}

\theoremstyle{plain}
\newtheorem{theorem}{Theorem}[section]

\newtheorem{lemma}[theorem]{Lemma}
\newtheorem{corollary}[theorem]{Corollary}

\newtheorem*{claim*}{Claim}
\theoremstyle{definition}

\newtheorem{definition}[theorem]{Definition}
\newtheorem{example}[theorem]{Example}

\DeclareMathOperator{\Endo}{End}
\DeclareMathOperator{\Gal}{Gal}

\def\F{\mathbb{F}}
\def\N{\mathbb{N}}
\def\Z{\mathbb{Z}}
\def\Q{\mathbb{Q}}
\def\R{\mathbb{R}}

\def\End{\operatorname{End}}

\newcommand{\Fq}{\mathbb{F}_q}

\newcommand{\closedFq}{\overline{\Fq}}

\newif\ifshowold %flag to set things visible or not 
\showoldfalse %\importanttrue (visible) or \importantfalse (invisible)
\begin{document}
%%%%%%%%%%%%%%%%%%%%%% TEMPORARY TITLE %%%%%%%%%%%%%%%%%
\title{On Isomorphic $K$-Rational Groups of Isogenous Elliptic Curves Over Finite Fields}

\author[1]{Liljana Babinkostova}
\address{Boise State University}
\email{liljanababinkostova@boisestate.edu}

\author[2]{Andrew Gao}
\address{University of California, Berkeley}
\email{agao@berkeley.edu}

\author[3]{Ben Kuehnert}
\address{University of Rochester}
\email{bkuehner@u.rochester.edu}

\author[4]{Geneva Schlafly}
\address{University of California, Santa Barbara}
\email{gschlafly@ucsb.edu}

\author[5]{Zecheng Yi}
\address{Johns Hopkins University}
\email{zeyi@jhu.edu}

%%%%%%%%%%%%%%%%%%%%%%%%%%%%%%%%%%%%%%%%%%%%%%%%%%%%%%%%%%%%%%
%                    INCLUDE YOUR NAMES HERE                 %
%%%%%%%%%%%%%%%%%%%%%%%%%%%%%%%%%%%%%%%%%%%%%%%%%%%%%%%%%%%%%%

% needs to be fixed- 6/30
\thanks{{\bf Acknowledgements:}  The research was provided by the National Science Foundation under the grant DMS-1062857 and Boise State University.}
%\thanks{$^{\S}$ Corresponding Author: Liljana Babinkostova}

\maketitle
\begin{abstract}
It is well known that two elliptic curves are isogenous if and only if they have same number of rational points. In fact, isogenous curves can even have isomorphic groups of rational points in certain cases.
In this paper, we consolidate all the current literature on this relationship and give a extensive classification of the conditions in which this relationship arises. First we prove two ordinary isogenous elliptic curves have isomorphic groups of rational points when they have the same $j$-invariant. Then, we extend this result to certain isogenous supersingular elliptic curves, namely those with equal $j$-invariant of either 0 or 1728,
using Vl\u{a}du\c{t}'s characterization of the group structure of rational points. We finally introduce and use a method by Heuberger and Mazzoli to establish a general case of this relationship within isogenous elliptic curves not necessarily having equal $j$-invariant, and discuss the required parameters for their method.

\end{abstract}
\section{Introduction}
Elliptic curves have long been a popular subject of interest in number theory due to their rich algebraic structure and their unique morphisms known as isogenies. In 1966, Sato and Tate proved that two elliptic curves defined over the same field are isogenous if and only if they have the same number of rational points in \cite{Tate}. This theorem offers a helpful tool to relate the study of multiple elliptic curves' algebraic structures to that of the isogenies between them.

The group structure of the rational points of an elliptic curve over a finite field is a central area of study of the algebraic structure of these objects as abelian varieties. In 1999, Vl\u{a}du\c{t} in \cite{Vladut} describes the situations in which these elliptic curves have a cyclic group of rational points based on its trace of Frobenius and the finite field over which the elliptic curve is defined. In 2001, Wittmann determines the possible group structures of $E(\F_{q^r})$ for all elliptic curves defined over $\F_q$ with isomorphic group of rational points. We will combine these two results to characterize the differences in the group structure of rational points of an elliptic curve based on the field that elliptic curve is defined over. Specifically, for an elliptic curve defined over $\F_{p^r}$, we will show which different structures $E(\F_{p^r})$ can have depending on whether $r$ is odd or even. This allows us to link the group structures of $E(\F_{q^r})$ and $E(\F_{q^s})$ when $r$ and $s$ have different parity. In this context, we also condense Vl\u{a}du\c{t}'s work to analyze the occurrence of both cyclic and non-cyclic groups of rational points. 

With this classification of these groups of rational points at hand, we then study the situations in which two isogenous curves have isomorphic groups of rational points. We first give a proof that two elliptic curves with the same $j$-invariant are isogenous if and only if they have isomorphic groups of rational points in the ordinary case. We then extend this statement to isogenous supersingular elliptic curves when their $j$-invariants are either 0 or 1728. Finally, we give a brief introduction to the tool introduced by Heuberger and Mazzoli in \cite{Heuberger} to decide whether two elliptic curves have isomorphic groups of rational points if their cardinalities are same. This method does not require two elliptic curves to have same $j$-invariant, and thus it can be used to classify groups of rational points in the more general case. However, their method requires using additional parameters. Here we analyze the utility of certain parameters and their consequences.

\section{Preliminaries} \label{SectionPrelim}
%%%%%%%%%%%%%%%%%%%%%%%%%%%%%%%%%%%%%%%%%%%%%%%%%%%%%%%%%%%%%%%%%%%% 
%  Start with introduction to elliptic curves and include affine and
%  projective planes, and notation - LB
%%%%%%%%%%%%%%%%%%%%%%%%%%%%%%%%%%%%%%%%%%%%%%%%%%%%%%%%%%%%%%%%%%%%
%\subsection{Notation}
%\begin{itemize}
%    \item \#A will denote the cardinality of $A$.
%    \item $\overline{K}$ is the algebraic closure of a field $K$.
%    \item $p$ prime, $q = p^r$ where $r\in\Z_{\geq 1}$.
%    \item $\text{End}(E) = \text{End}_{\overline{\F q}}(E)$
%    \item ${\text{End}_{K}(E)} = {\text{Hom}_{K}(E,E)}$
%    \item Frobenius Map: $\pi_{p^r} = \pi_p^r$
%    \item $K^\times$ will denote the multiplicative group for a field $K$.
%    \item The $K$-rational group of $E$ is $E(K)$.
%\end{itemize}
%%%%%%%%%%%%%%%%%%%%%%%%%%%%%%%%%%%%%%%%%%%%%%%%%%%%%%%%%%%%%%%%%%%%%%%
% State Hasse's theorem. Define the addition law; state the abelian group % theorem; include a graph that demonstrates the addition law. - LB
%%%%%%%%%%%%%%%%%%%%%%%%%%%%%%%%%%%%%%%%%%%%%%%%%%%%%%%%%%%%%%%%%%%%%%%
\subsection{Elliptic Curves} 
%\subsection{Affine and Projective Planes}

We introduce some elementary features of elliptic curves. Let $K$ be a field and $\overline{K}$ be its algebraic closure. An elliptic curve $E$ is an algebraic curve defined by a minimal Weierstrass equation
\begin{align*}
    y^2 + a_1xy + a_3y = x^3 + a_2x^2 + a_4x + a_6
\end{align*}
where $a_i \in K$. In this paper, fields will be assumed to have characteristic not 2 or 3, and in this case we can use the Weierstrass normal form given by
\begin{align*}
	E: y^2 = x^3 + Ax + B,
\end{align*} 
with $A,B\in K$. An elliptic curve is a nonsingular algebraic curve, and a curve is nonsingular if the discriminant $\Delta = -16(4a^3 + 27B^2)$ is nonzero. 
%Often it is useful to consider these curves in the projective plane. In order to do so, we homogenize $E$, giving
%$$E:y^2z=x^3+Axz^2+Bz^3.$$
%Notice that the points of the form $[x:y:1]$ correspond to the points on the affine version of $E$. Points of the form $[x:y:0]$ are called ``points at infinity'' in projective geometry, and in this case, the points of infinity on the curve $E$ are given by
%$$y^2\cdot 0 = x^3+Ax\cdot 0^2 + B\cdot 0^3$$
%So, $x^3=0$, and $y$ can be any number. Thus, the points at infinity intersecting with $E$ are the points $[0:y:0]$, or simply $[0:1:0]$, as we work modulo scalar multiples. This point will be denoted the ``point at infinity''.

%Then affine and projective planes are important to define operations on elliptic curves. First, affine $n$-space over a field $K$ is simply $\A^n(K)=\{P\in K^n\}$ and the affine plane is the special case $\A^2(K)$. Next, the projective plane is given by $\P^2(K)=\A^3(K)$ with the equivalence relation that
%$$[x:y:z] ~ [\lambda x : \lambda y : \lambda z] \quad \text{for } \lambda\in K.$$
%Note, that in projective space, coordinates $(x,y,z)$ are written as $[x:y:z]$. Notice that the affine plane is embedded into the projective plane, as any $[x:y:z]=[x/z:y/z:1]$. 

% Commenting out bits about projective plane. We do not discuss alternative co-ordinate systems for this paper, and it seems to not be useful. 

The $j$-invariant of an elliptic curve is an important invariant of curves. For a curve $E:y^2=x^3+Ax+B$ in Weierstrass normal form, the $j$-invariant is given by
$$j(E)=1728\cdot \frac{4A^3}{4A^3+4B^2}.$$

For an elliptic curve $E$ defined over a field $K$, we will define two important sets. The first is 

$$E/K=E(\overline{K})=\{(x,y)\in\overline{K}^2: y^2=x^3+Ax+B\}\cup\{\infty\}$$
where the element $\infty$ is called the \textit{point at infinity}, and will serve as an identity element. Note that this set is infinite in general. The second set is the $K$-rational points, denoted
$$E(K)=\{(x,y)\in K^2 : y^2=x^3+Ax+B\}\cup\{\infty\}.$$
 These sets are both Abelian groups when equipped with the following group operation:

\subsubsection{Group Law}
For two points $P$ and $Q$ on an elliptic curve, the point $P+Q$ is computed by first calculating the line intersecting both $P$ and $Q$. This line will intersect the elliptic curve at a third point, call it $R$. Finally, draw a vertical line through $R$ and denote the point at which this line intersects the curve as $P+Q$. Points that share a vertical line are inverses, and when adding a point to itself, we use the tangent line. See Figure 1 for some examples of this for an elliptic curve defined over $\R$. \vspace{0.5cm}

\begin{figure}
    \centering
    \includegraphics[scale=1]{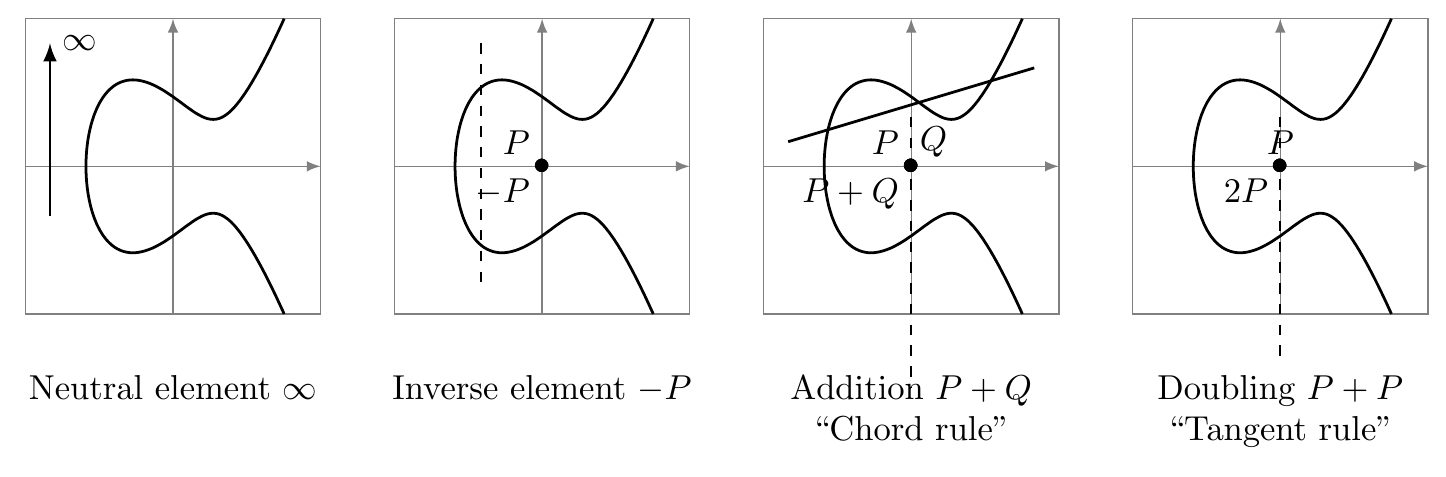}
    \caption{Group Law of Elliptic Curves}
    \label{fig:my_label}
\end{figure}

Note that $E(K)\subseteq E(\overline{K})$, in particular, $E(K)$ is a subgroup of $E(\overline{K})$. Let $q=p^r$ where $p$ is prime and $r\in\N$. When $K$ is the finite field $\F_q$, $E(K)$ is a finite abelian group and is classified by the following theorems:

\begin{theorem}[Hasse]
    Let $E$ be an elliptic curve over a finite field $\F_q$. Then,
    $$\#E(\F_q) = q+1-t$$
    where $|t|\leq 2\sqrt{q}$.
\end{theorem}

\begin{theorem}\label{later}
    Let $E$ be an elliptic curve over a finite field $\F_{q}$. Then, 
    $$E(\F_{q}) \cong (\Z/n_1\Z)\oplus(\Z/n_2\Z)$$
    where $n_1n_2=\#E(\F_q)$ and $n_1 | n_2$.
\end{theorem}
\noindent The proof of this theorem will be left later after we introduce the notion of torsion points in section \ref{endomorphism}. 
\begin{definition}
    Let $E$ be an elliptic curve over a finite field $\F_{p^r}$. We call $E$ \textbf{supersingular} if $\#E(\F_{p^r})\equiv 1 \mod p$. Elliptic curves that are not supersingular are called \textbf{ordinary} elliptic curves.
\end{definition}{}

%\begin{proof}

%%%%%%%%%%%%%%%%%%%%%%%%%%%%%%%%%%%%%%%%%%%%%%%%%%%%%%%%%%%%%%%%%%%%%%
% The proof of the theorem should be expanded. See Washington. _ LB  %
%%%%%%%%%%%%%%%%%%%%%%%%%%%%%%%%%%%%%%%%%%%%%%%%%%%%%%%%%%%%%%%%%%%%%%

%This is a direct consequence of the existence of the Weil Pairing $E(\Fq) \times E(\Fq) \to \Fq^*$. See Silverman, Corollary 8.1.1.
%\end{proof}

%%%%%%%%%%%%%%%%%%%%%%%%%%%%%%%%%%%%%%%%%%%%%%%%%%%%%%%%%%%%%%%%%%
% This should be a subsection of the Elliptic Curves section - LB
%%%%%%%%%%%%%%%%%%%%%%%%%%%%%%%%%%%%%%%%%%%%%%%%%%%%%%%%%%%%%%%%%%

\subsection{Maps Between Elliptic Curves}
    The maps between elliptic curves will be crucial in understanding their structure. Let $E$ and $E'$ be elliptic curves defined over a field $K$. Then, a rational function $f:E(\overline{K})\to E'(\overline{K})$ over $K$ is given by 
    $$f(x,y)=\left(\frac{f_0(x,y)}{f_1(x,y)},\frac{f_2(x,y)}{f_3(x,y)}\right)\quad f_i\in K[x,y]$$
    Alternatively, a rational function defined over $\overline{K}$ would consist of polynomials from $\overline{K}[x,y]$. 
    \par
        For notation purposes, we will denote rational functions as $f:E\to E'$, even though they are technically maps between $E(\overline{K})$ and $E'(\overline{K})$. This is to reinforce that these are maps of \textit{curves} rather than of groups.
    A \textbf{morphism} $f:E\to E'$ over $K$ is a rational function over $K$ that is defined everywhere. The following lemma gives a canonical form for all these morphisms. 
    \begin{lemma}
        Let $f:E\to E'$ be a morphism over $K$. Then, there exist $r_1,r_2,s_1,s_2\in K[x]$ such that
        $$f(x,y)=\left( \frac{r_1(x)}{s_1(x)},\frac{r_2(x)}{s_2(x)}y\right).$$
    \end{lemma}
    Additionally, all morphisms are either constant or surjective.
    
    \begin{definition}
        Let $E,E'$ be elliptic curves over $K$. A morphism $\alpha:E\to E'$ is an isogeny if is surjective and satisfies $\alpha(\infty_{E})=\infty_{E'}$.
    \end{definition}
     Isogenies play the role of homomorphisms on curves. In fact, an isogeny $\alpha:E\to E'$ defined over $K$ restricts to a group homomorphism $\alpha:E(K)\to E'(K)$. Hence, we will call two curves $E$ and $E'$  isomorphic over $K$ if there exists a pair of isogenies defined over $K$  $\phi: E\to E'$ and $\psi: E' \to E$ such that $\psi \circ \phi = [1]_{E}$ and $\phi \circ \psi =[1]_{E'}$ where $[1]_E$ is the identity map on $E$. Similarly, elliptic curves are isomorphic over $\overline{K}$ if the pair of isogenies are defined over $\overline{K}$. The following theorem is an example of an explicit isogeny over $\overline{K}$.
     
     \begin{theorem}
         Let $E_1:y_1^2=x_1^3+A_1x_1+B_1$ and $E_2:y_2^2=x_2^3+A_2x_2+B_2$ be elliptic curves defined over $K$. If $j(E_1)=j(E_2)$ there exists a $\mu \in \overline{K}$ with
         $$A_2=\mu^4A_1, \quad B_2=\mu^6B_1$$
         such that the transformation
         $$x_2=\mu^2x_1, \quad y_2=\mu^3y_1$$
         is an isomorphism over $\overline{K}$.
     \end{theorem}
     
     \begin{proof}
        See Theorem 2.19 in \cite{Wl}.
     \end{proof}
     It is also the case that the converse is true. If two elliptic curves are isomorphic over $\overline{K}$, then they must have the same $j$-invariant.

    Next, we introduce the notion of the dual and show that isogenies actually form equivalence relations. From this, a theorem of Sato and Tate shows that isogeny classes of curves are in one-to-one correspondence with the orders of the curves' $K$-rational points.
    \begin{theorem}
        For every isogeny $\phi:E\to E'$, there exists a unique isogeny $\hat{\phi}:E'\to E$, called the dual that is an isogeny itself.
    \end{theorem}
    
    \begin{corollary}
        Isogenies form an equivalence relation on elliptic curves.
    \end{corollary}
    
    \begin{proof}
        Isogenies satisfy the required properties:
        \begin{enumerate}
            \item Reflexive: Every elliptic curve is isogenous to itself via the identity map.
            \item Symmetric: The existence of dual isogenies imply that if $E$ is isogenous to $E'$, then $E'$ is isogenous to $E$.
            \item Transitive: Since isogenies are surjective, they can be composed to create new isogenies.
        \end{enumerate}
    \end{proof}
    
    \begin{theorem}[Sato-Tate's Isogeny Theorem (1966)] \label{thm: tate}
    Two elliptic curves $E_1/K$ and $E_2/K$ are isogeneous if and only if 
    \[\# E_1(K) = \# E_2(K). \]
    \end{theorem}

    %\begin{remark}
    %Silverman emphasizes that for general varieties, the %$\phi$ and $\psi$ above must be morphisms, not merely %rational maps. By definition, all isogenies are %morphisms between algebraic varieties, thus isogenies %satisfy the criteria above.
  %  \end{remark}
    As mentioned previously, isogenies defined over $K$ induce group homomorphisms between the elliptic curves' $K$-rational points. The following theorem shows that invertible isogenies over $K$ induce isomorphisms of $K$-rational points.

    \begin{theorem} \label{thm: iso E/K to E(K)}
        If $E_1$ and $E_2$ are elliptic curves isomorphic over $K$, then $E_1(K)\cong E_2(K)$.
    \end{theorem}
    
    \begin{proof}
    Let $\phi:E\to E'$ be an isomorphism over $K$. Then, since $E(K)\leq E(\overline{K})$, $\phi[E(\overline{K})]\leq E'(\overline{K})$. In particular, since $\phi$ takes $K$-rational points to $K$-rational points, as it is defined over $K$, we have 
    \begin{equation}  \phi[E(K)] \leq E'(K)
    \label{eqn: iso for}
    \end{equation}
    This same argument holds for the inverse, $\phi^{-1}:E'\to E$. So,
    \begin{equation*}\phi^{-1}[E'(K)] \leq E(K).\end{equation*}
    Now, applying $\phi$ to the inequality above, we have
    
    \begin{equation}E'(K) \leq \phi[E(K)].
    \label{eqn: iso rev}
    \end{equation}
    Thus, by (\ref{eqn: iso for}) and (\ref{eqn: iso rev}), $\phi[E(K)]$ and $E'(K)$ are subgroups of each other, and thus 
    $$\phi[E(K)] \cong E'(K).$$
    Finally, by the First Isomorphism Theorem,  $\phi$ is an isomorphism between $E(K)$ and $E'(K)$. 
    \end{proof}

    \par We will now give two tools to detect if isogenies are defined over $K$.
    \begin{lemma}
        Let $\sigma\in\Gal(\overline{K}/K)$ and $\phi:E\to E'$ be an isogeny. Then $\phi$ is defined over $K$ if and only if $\phi^\sigma = \phi$
        for all $\sigma$.
    \end{lemma}
    \begin{lemma}
        An isogeny $\phi:E\to E'$ is an isogeny is defined over $K$ if the image of $E(K)$ under $\phi$ is contained in $E'(K)$.
    \end{lemma}

\subsection{Endomorphisms}\label{endomorphism}
    An endomorphism of an elliptic curve $E$ is an isogeny from $E$ to itself. The set $\End_{K}(E)$ of endomorphisms over $K$ has a ring structure with the following operations. Let $E/K$ be an elliptic curve and let $\alpha,\beta \in \Endo_K(E)$. Then, $\alpha+\beta$ will be the pointwise addition of functions. So for $P\in E/K$, $(\alpha+\beta)(P)=\alpha(P)+\beta(P)$. Multiplication will be given by function composition, so $\alpha\cdot\beta = \alpha\circ\beta$. The structure is exactly the same for the ring $\End_{\overline{K}}(E)$ of endomorphisms defined over $\overline{K}$.
    \par
        An example of an endomorphism is the \textbf{multiplication by }$n$\textbf{ map}, denoted by $[n]$ where $n$ is a nonzero integer. For an elliptic curve $E/K$ and a point $P\in E/K$, the map is defined via
        $$[n]P = nP = \underbrace{P + \cdots + P}_\text{$n$ times}$$
        Note that $[0]$ maps everything to zero, so it is constant, and not an isogeny by definition. This map has some noteworthy properties:
        \begin{itemize}
            \item $[n+m]=[n]+[m]$ as $(n+m)P=nP+mP$
            \item $[nm]=[n]\circ[m]$ as $(nm)P=n(mP)$.
        \end{itemize}
        With the multiplication by $n$ map defined, we will also introduce the definition of torsion points here. Given an elliptic curve $E/K$, $E[n]:=\{P\in \overline{K} \mid  nP=\infty\}$, that is $E[n]$ is the set of point that goes to infinity under the multiplication by $n$ map. We also have the following theorem concerning the group structure of $E[n]$:
        
        \begin{theorem}[Theorem 3.2 in \cite{Wl}]\label{to prove former theorem}
            Let $E $ be an elliptic curve defined over $K$ and $n$ be a positive integer. If the characteristic of $K$ does not divide $n$ or is 0, then
            $$
            E[n]\cong \Z/n\Z \oplus \Z/n\Z.
            $$
            If the char($K$)=$p>0$ and $p|n$, then we can write $n=p^r n'$ with $p\nmid n'$. Then
            $$
            E[n] \cong \Z/n'\Z \oplus \Z/n'\Z \text{ or } \Z/n\Z\oplus \Z/n'\Z.
            $$
        \end{theorem}
        \noindent Now we are ready to give the proof of theorem \ref{later}.
        \begin{proof}[Proof of Theorem \ref{later}]
        From group theory, we know that a finite abelian group can be written into a direct sum of cyclic groups of the form 
        $$
        \Z/n_1\Z \oplus \Z/n_2\Z \oplus \cdots \oplus \Z/n_r\Z
        $$
        with $n_i|n_{i+1}$ for $i\geq 1$. For each $n_i$, $\Z/n_i\Z$ contains a one copy $\Z/n_1\Z$ and thus have $n_1$ elements of order dividing $n_1$. So in total we get $n_1^r$ elements in $E(\F_q)$ with order dividing $n_1$. But by Theorem \ref{to prove former theorem}, there can be at most $n_1^2$ such elements, so we get $r=2$, which proves the theorem. 
        \end{proof}
        
        Another important endomorphism is the \textbf{Frobenius endomorphism}. Let $E/\F_q$. Then, the Frobenius endomorphism is defined by
        \begin{align*}
            \pi_q:& E/\F_q\to E/\F_q\\ &(x,y)\mapsto (x^q,y^q)
        \end{align*}

        \begin{lemma}[Lemma 4.5 in \cite{Wl}]
            Let $E/\F_q$ be an elliptic curve, and let $(x,y)\in E(\closedFq)$. Then,
            \begin{enumerate}
                \item $\pi_q(x,y)\in E(\closedFq) $ and
                \item $(x,y)\in E(\F_q)$ if and only if $\pi_q(x,y)=(x,y)$.
            \end{enumerate}
        \end{lemma}
        
        \begin{corollary}
            If an endomorphism $\phi\in\End_{\overline{\F_q}}(E)$ commutes with the Frobenius endomorphism $\pi_q$, then $\phi\in\End_{\F_q}(E)$.
        \end{corollary}
        \begin{proof}
            Suppose $\phi$ commutes with $\pi_q$. Let $P$ be a point in  $E(\F_q)$. Then,
            
            $$\pi_q\circ \phi(P)=\phi\circ \pi_q(P).$$
            
            Since $P\in E(\F_q)$, then $\pi(P)=P$ by Lemma 2.12. Hence,  $$(\phi(P))^q=\phi(P^q)=\phi(P).$$
            
            Again, by Lemma 2.12, this means that $\phi(P)\in E(\F_q)$. Therefore, $\phi$ restricted to $E(\F_q)$ maps into $E(\F_q)$, and is thus defined over $\F_q$ by Lemma 2.11.
        \end{proof}
    
\section{Ordinary Elliptic Curves with Equal J-Invariant}
In this section we will discuss the relationship between isogenous elliptic curves' $F_q$-rational points with equal  $j$-invariant. 

\begin{theorem}[\cite{Silverman}]
    If $E/\F_q$ is an ordinary elliptic curve, then $\End_{\overline{\F_q}}(E)$ is an order in an imaginary quadratic field, and hence commutative.
\end{theorem}

\begin{corollary}
    If $E/\F_q$ is ordinary, then $\End_{\overline{\F_q}}(E)=\End_{\F_q}(E)$.
\end{corollary}{}

\begin{proof}
    Let $\phi\in\End_{\overline{\F_q}}(E)$. Since $\End_{\overline{\F_q}}(E)$ is commutative, then $\phi$ commutes with $\pi_q$, hence $\phi\in\End_{\F_q}(E)$. Thus, $\End_{\overline{\F_q}}(E)\subseteq \End_{\F_q}(E)$. Trivially, $ \End_{\F_q}(E)\subseteq \End_{\overline{\F_q}}(E)$ since all isogenies defined over $\F_q$ are also defined over $\overline{\F_q}$. Thus,
    $$\End_{\overline{\F_q}}(E) = \End_{\F_q}(E)$$
\end{proof}{}

\begin{theorem} \label{thm: ord iso}
    Suppose $E_1$ and $E_2$ are ordinary elliptic curves defined over $\F_q$ with $j(E_1)=j(E_2)$. Then,
    $$\#E_1(\F_q)=\#E_2(\F_q) \iff E_1(\F_q)\cong E_2(\F_q)$$
\end{theorem}

\begin{proof}
    We proceed as Proposition 14.19 in \cite{Cox}. Suppose $E_1(\F_q)\cong E_2(\F_q)$, then $\#E_1(\F_q)=\#E_2(\F_q)$, as isomorphisms of finite groups preserve order. Conversely, suppose $\#E_1(\F_q)=\#E_2(\F_q)$. By Sato-Tate, these curves are isogenous over $\F_q$, so there exists some isogeny $\lambda:E_1\to E_2$ defined over $\F_q$. Next, since $j(E_1)=j(E_2)$, there exists an isomorphism $\phi:E_2\to E_1$ defined over $\overline{\F_q}$. Consider $\phi\circ\lambda$ which is in $\End_{\overline{\F_q}}(E_1)$. Since $E_1$ is ordinary, this endomorphism can be defined over $\F_q$, so $\phi\circ\lambda\in\End_{\F_q}(E_1)$. Let $\sigma\in \Gal(\overline{\F_q}/ \F_q)$. Then,
    $$\phi^\sigma\circ \lambda = \phi^\sigma\circ\lambda^\sigma = (\phi\circ\lambda)^\sigma = \phi\circ\lambda$$
    Since isogenies are surjective, then $\phi^\sigma=\phi$. Since this holds for all such $\sigma$, then $\phi$ is defined over $\F_q$. Finally, since $E_1$ and $E_2$ are isomorphic over $\F_q$, then $E_1(\F_q)\cong E_2(\F_q)$ via Theorem 2.9. 
    %Need that ordinary endomorphism rings are commutative.
    %Need def of defined over F_q via galois automorphisms
    %Need isogenous over K implies K-rational points are isomorphic
    %Need basic galois theory in prelims
    %Cite cox
\end{proof}

\section{J-Invariant Zero Case}
We prove that when the $j$-invariant is zero, two curves are in the same isogeny class if and only if they have isomorphic $\F_q$-rational points. Since saying two curves are in the same isogeny class is equivalent to saying they have the same number of $\F_q$ rational points, we will prove that two curves have the same number of $\F_q$-rational points if and only if the group of $\F_q$-rational points are isomorphic.\vspace{5mm}

We show the conditions for which an elliptic curve with $j$-invariant zero is supersingular.

\begin{theorem}\label{order2mod3}
For $p\equiv 2\mod3$, if $E$ is an elliptic curve with $j$-invariant 0 defined over $F_q$, then:
\begin{enumerate}
    \item If $q$ is an odd power of $p$, then $\#E(\F_q)=q+1$;
    \item If $q$ is an even power of $p$, then $\#E(\F_q)=q+1+\{\pm 2\sqrt{q},\pm \sqrt{q}\}$.
\end{enumerate}
\end{theorem}

\begin{theorem}[Gauss] \label{thm: gauss}
When $p\equiv 1 \mod 3$, if $E:y^2=x^3+r$ is an elliptic curve defined over $\F_p$, then 
$$
\#E(\F_p)=\begin{cases}
p+1+2a & \text{if } r \text{ is sextic residue modulo } p\\
p+1-2a & \text{if } r \text{ is cubic residue but not quadratic residue modulo } p\\
p+1-a\pm3b & \text{if } r \text{ is quadratic residue but not cubic residue modulo } p\\
p+1+a\pm 3b & \text{if } r \text{ is neither quadratic nor cubic residue modulo } p,
\end{cases}
$$
where $a,b$ satisfy the conditions $a^2+3b^2=p$, $b>0$, and $a\equiv2\mod 3$.
\end{theorem}

\begin{definition}
The \textbf{trace of Frobenius} of an elliptic curve $E/\F_q$, usually denoted as $m$, is defined as $m=q+1-\#E(\F_q)$. 
\end{definition}

\begin{lemma}\label{ordinary}
An elliptic curve $E: y^2=x^3+r$ defined over $\F_p$ is ordinary if $p\equiv 1 \mod 3$. 
\end{lemma}
\begin{proof}
Here we consider the relationship between $p$ and the trace of Frobenius, $m=p+1-\#E(\F_p)$. By Theorem \ref{thm: gauss}, when $p\equiv 1 \mod 3$, we have $m\in\{ \pm2a,-a\pm3b,a\pm3b\}$ where $a\equiv 2\mod 3$ and $b>0$. These conditions give us the following inequalities:
\begin{align*}
    |2a| &\leq a^2 = p-3b^2 < p,\\
    |-a\pm3b| &\leq |a|+|3b| < a^2+3b^2 = p,\\
    |a\pm3b| &\leq |a|+|3b| < a^2+3b^2 = p.
\end{align*}
Thus, $p$ does not divide any of the elements in $\{\pm2a,-a\pm3b,a\pm3b\}$.
Hence, $p$ does not divide the trace of Frobenius.
Therefore, when $p\equiv 1\mod 3$, $E$ is ordinary.
\begin{comment}
$$
    |2a|\leq a^2=p-3b^2<p.
$$
So $p\nmid 2a$ and $p\nmid -2a$. Besides, we also have
$$
    |-a\pm3b|\leq|a|+|3b|<a^2+3b^2=p
$$
and 
$$
    |a\pm3b|\leq|a|+|3b|<a^2+3b^2=p,
$$
which implies $p$ does not divide any of the elements in $\{2a,-2a,-a\pm3b,a\pm3b\}$.

Since $p$ does not divide the trace of Frobenius when $p\equiv 1\mod 3$, every elliptic curve defined over $F_p$ is ordinary. 
\end{comment}
\end{proof}

\begin{lemma}\label{extension}
For an elliptic curve $E: y^2=x^3+r$ defined over $\F_q$ with $r\in \F_p^{\times}$, if $E$ is ordinary over $\F_p$, $E$ is also ordinary over $\F_q$.
\end{lemma}
\begin{proof}
Since $E$ is ordinary over $\F_p$, we have $E[p]\cong \Z/p\Z$. Since $\F_q$ also has characteristic $p$, $E[p]\cong \Z/p\Z$ implies $E$ is ordinary over $F_q$.
\end{proof}

\begin{theorem}\label{os}
An elliptic curve $E$ defined over $\F_q$ with $j(E)=0$ is supersingular if and only if $p\equiv 2\mod 3$.
\end{theorem}
\begin{proof}
If $p\equiv 2\mod 3$, by theorem \ref{order2mod3} 
$$
m=q+1-\#E(\F_q)=
\begin{cases}
0 & \text{if } q \text{ is an odd power of }p\\
\pm2\sqrt{q} \text{ or } \pm\sqrt{q} & \text{if } q \text{ is an even power of }p.
\end{cases}
$$
Thus when $q$ is odd power of $p$, $p\mid m=0$, $E$ is supersingular. When $q$ is even power of $p$, $\sqrt{q}$ will still be a power of $p$, so $E$ is also supersingular. \\
\indent Now suppose $p\not\equiv 2\mod 3$, $i.e.$ $p\equiv 1\mod 3$. For any elliptic curve $E:y^2=x^3+r$, there must exist an $r'\in \F_p^{\times}$ such that $r$ and $r'$ are in the same sextic class over $F_q$. So we know $y^2=x^3+r$ and $y^2=x^3+r'$ have the same number of rational points, and thus have the same trace of Frobenius, $i.e.$ $m_r=m_{r'}$. According to lemma \ref{ordinary} and lemma \ref{extension}, $p\nmid m_{r'}$ since $r'\in \F_p^{\times}$, thus $p\nmid m_r$ and $E$ is ordinary over $\F_q$. 
\end{proof}
\vspace{3mm}
Now since we know an elliptic curve with $j$-invariant 0 is supersingular if and only if $p\equiv 2\mod 3$ (thus is ordinary if and only if $p\equiv 1\mod 3$), we will discuss the relationship between order and group structure of $\F_q$-rational points with respect to $p$.

\subsection{Group Structure of Supersingular Curves}

\begin{theorem}[Theorem 2.1 in \cite{Vladut}]\label{major}
A finite abelian group $G$ of order $N=q+1-m$, with $m^2\leq 4q$, is isomorphic to $E(\F_q)$ for $E$ over $\F_q$ if and only if one of the following conditions holds:
\begin{enumerate}
\item $p$ does not divide $m$, and $G\cong \Z/A\times \Z/B$, where $B|A$ and $B|(m-2)$.
\item $q$ is an odd power of $p$ and one of the following holds:
\begin{enumerate}
\item $m=0$, $p\equiv 1,2\mod 4$, and $G$ is cyclic. \label{subthm: ord cyclic}
\item $m=0$, $p\equiv 3\mod 4$, and $G$ is either cyclic or $G\cong \Z/(q+1)/2 \times \Z/2$. \label{subthm: ss cyclic}
\item $p=2$ or 3, $m=\pm\sqrt{pq}$, and $G$ is cyclic.
\end{enumerate}

\item $q$ is an even power of $p$ and one of the following holds:
\begin{enumerate}
\item $m=\pm2\sqrt{q}$ and $G\cong (\Z/(\sqrt{q}\mp 1))^2$.
\item $m=\pm\sqrt{q}$, and $p=3$ or $p\equiv 2\mod 3$, and $G$ is cyclic.
\item $m=0$, and $p\equiv 2,3\mod4$, and $G$ is cyclic. 
\end{enumerate}
\end{enumerate}

\end{theorem}

\begin{theorem}
If $E_1$ and $E_2$ are defined over $\F_q$ and $p\equiv 2 \mod 3$, then $\#E_1(\F_q)=\#E_2(\F_q)$ if and only if $E_1(F_q)\cong E_2(F_q)$.
\end{theorem}
\begin{proof}
One direction of the theorem is very straightforward, so we will only show that same number of $\F_q$ rational points implies same group structure of $\F_q$ rational points.\\
\indent Firstly, we consider the case when $q$ is an odd power of $p$. By theorem \ref{order2mod3}, 
$$
    \#E(\F_q)=q+1 \text{ for all } E/\F_q.
$$
So we have the trace of Frobenius, denoted by $m$, is given by 
$$m=q+1-\#E(\F_q)=0.$$
Since $p\equiv 2\mod 3$, we can either have $p\equiv 1\mod 4$ or $p\equiv 3\mod 4$, which corresponds to (\ref{subthm: ord cyclic}) and (\ref{subthm: ss cyclic}) in Theorem \ref{major}. In the first case, $E(\F_q)$ is cyclic. In the latter case, $E(\F_q)$ is either cyclic or isomorphic to $\Z/(q+1)/2\times Z/2$ according to the theorem. However, since we know $E[2]\not\subseteq E(\F_p)$, by Theorem 4.1 in \cite{Wittmann}, $E(\F_q)$ can only be cyclic. 
Therefore, since they all have the same order, we know 
$$
    E_1(F_q)\cong \Z/(q+1)\Z\cong E_2(F_q).
$$

Now suppose instead that $q$ is an even power of $p$. Then according to theorem \ref{order2mod3}, 
$$
    m=q+1-\#E(\F_q)=\pm 2\sqrt{q} \text{ or }\pm \sqrt{q}.
$$ 
By using Theorem \ref{major}(3)(b), if $m=\pm\sqrt{q}$, $E(\F_q)$ is again cyclic, $i.e.$ 
$$
    E_1(\F_q)\cong A\cong E_2(\F_q) ,
$$  
where 
$$
    A\in \{  \Z/(q+1+\sqrt{q})\Z,\  \Z/(q+1-\sqrt{q})\Z\}.
$$
Likewise, Theorem \ref{major}(3)(c) says that if $m=\pm2\sqrt{q}$, then
$$
    E_1(\F_q)\cong B\cong E_2(\F_q),
$$
where 
$$
    B\in\{\Z/(\sqrt{q}-1)\times \Z/(\sqrt{q}-1),\ \Z/(\sqrt{q}+1)\times \Z/(\sqrt{q}+1) \}.
$$
Therefore, when two curves have the same number of $\F_q$ rational points, their group of $\F_q$ rational points are always isomorphic.
\end{proof}

\vspace{5mm}
\subsection{Group Structure of Ordinary Curves}~{}\\

According to Theorem \ref{os}, we know that $E$ is ordinary if $p\equiv 1\mod 3$. Therefore, our results for ordinary elliptic curves will hold here. We conclude that for elliptic curves $E_1,E_2$ defined over $\F_q$, if their $\F_q$-rational points have the same cardinality, then $E_1(\F_q)\cong E_2(\F_q)$.

\section{J-invariant 1728 Case}

Let $E$ be elliptic curves over field $K$ with j-invariant 1728. That means $E$ is of the form
\begin{equation}
    y^2 = x^3 + ax,
\end{equation}
where $a\in K$. Its automorphism group $Aut(E)$ has order 4, with one automorphism mapping $(x,y)$ to $(-x,iy)$.

\begin{theorem} \label{thm: 1728 order}
    Let $E$ be an elliptic curve defined over $\F_q$ and $p \equiv 3\mod 4$.
    \begin{enumerate}
        \item If $q$ is an odd power of $p$, then $\#E(\F_q) = q+1$.
        \item If $q$ is an even power of $p$, then $\#E(\F_q) \in \{q+1 \:, \:q+1\pm 2\sqrt{q}  \}$
    \end{enumerate}
\end{theorem}
\begin{proof}
    See Theorem 3.6 and 3.7 in \cite{Kim}.
\end{proof}

% \begin{theorem}\label{order1}
%     For elliptic curve $E:y^2=x^3+Ax$ defined over $\F_{p^r}$, if $p\equiv 3\mod 4$ and $r$ is odd, $\#E(\F_{p^r})=p^r+1$.
% \end{theorem}
% \begin{proof}
%     See Theorem 3.6 in \cite{Kim}.
% \end{proof}

% \begin{theorem}\label{order2}
% For elliptic curve $E:y^2=x^3+Ax$ defined over $\F_{p^r}$, if $p\equiv 3\mod 4$ and $r$ is even, $\#E(\F_{p^r})\in \{p^r+1,p^r+1\pm 2\sqrt{p^r}  \}$.
% \end{theorem}
% \begin{proof}
%     See Theorem 3.7 in \cite{Kim}.
% \end{proof}

\begin{theorem}\label{soro}
    Let $E$ be an elliptic curve defined over $\F_q$ with $j$-invariant 1728.
    $E$ is supersingular if and only if $p\equiv 3\mod 4$.
\end{theorem}
\begin{proof}
Suppose $p\equiv 3\mod 4$. By Theorem \ref{thm: 1728 order}, we calculate the trace of Frobenius $m$, 
\[
m=\begin{cases}
0 & \text{ if }q \text{ is an odd power of }p\\
0 \text{ or } \pm 2\sqrt{q} & \text{ if }q \text{ is an even power of }p.
\end{cases}
\]
Since $p$ divides $m$ when $m=0$, and $p$ divides $2\sqrt{q}$ when $q$ is an even power of $p$, $E$ is supersingular. \\
\indent Now suppose $p\equiv 1\mod 4$. Use a similar technique to the proof of Theorem \ref{os}, by applying the Gauss's Theorem for $E: y^2=x^3+ax$ and extend the result from $\F_p$ to $\F_q$. This shows $E$ is ordinary. 
\end{proof}

\vspace{0.4cm}
\subsection{Group Structure of Supersingular Curves}~{}\\
By Theorem \ref{soro}, saying a elliptic curve defined over $\F_{q}$ with $j$-invariant 1728 is supersingular is equivalent to saying $p\equiv 3\mod 4$. Since the we considered the trace of Frobenius in two cases, namely when $q$ is an odd power of $p$ and when $q$ is an even power of $p$, we may also discuss the group structure of $\F_q$ rational points separately according to these two cases.

\vspace{0.4cm}
\subsubsection{When $q$ is an odd power of $p$}~{}\\
The group structure of $\F_q$ rational points in this case corresponds to Theorem \ref{major} (\ref{subthm: ss cyclic}), from which we can see that the there are two possible group structures for $E(\F_q)$. However, when $q$ is an odd power of $p$, the order of $E(\F_q)$ is uniquely given by $\#E(\F_q)=q+1$. So two elliptic curves having the same number of $\F_q$-rational points does not imply their group of $\F_q$-rational points are isomorphic. Example \ref{ex: supersing fail} shows an example of this statement.

\begin{example} \label{ex: supersing fail}
    Over $\F_7$, there are six elliptic curves all of order 8. So, all the curves are isogeneous by Sato-Tate Theorem, forming one isogeny class. But, there are two isomorphism classes with distinct group structures. In this case these are $\znz{8}$ and $\znzznz{2}{4}$. See the cases when $j=1728$ and $r$ is odd in Appendix for more examples.
\end{example}

On the other hand, in the following example where $q$ is an even power of $p$, same order of $\F_q$ rational points does imply the isomorphic group structure. In fact, this is true whenever $q$ is an even power of $p$ and we will prove this statement in next section of this type.
\begin{example}
    Over $\F_{7^2}$, there are 48 elliptic curves of three distinct orders, forming three isogeny classes, by Sato-Tate Theorem. 
    But, there are four isomorphism classes. 
    All the curves within the same isogeny class have the same group structure. However, two of the isomorphism classes contain curves of order 50. See the cases when $j=1728$ and $r$ is even in Appendix for more examples of this type.
\end{example}

\vspace{0.4cm}
\subsubsection{When $q$ is an even power of $p$}

\begin{theorem}
    If $E_1,E_2$ be supersingular elliptic curves over $\F_q$ with $r$ even and $j$-invariant 1728, then
 $$
        \# E_1(\F_{q}) = \# E_2(\F_{q}) \iff E_1(\F_{q}) \cong E_2(\F_{q}).
 $$
\end{theorem}

\begin{proof}
    $E_1(\F_{q}) \cong E_2(\F_{q})$ implies $\# E_1(\F_{q}) = \# E_2(\F_{q})$ by properties of isomorphic groups. So we only need to show that same order of $\F_q$ rational points imply the same group structure. 
    
    In proof of Theorem \ref{soro}, we showed that when $p\equiv 3\mod 4$ and $q$ is an even power of $p$, $m\in \{ 0, 2\sqrt{q}, -2\sqrt{q}\}$.
    When $m=0$, $\#E_1(\F_q)=\#E_2(\F_q)=q+1$. By Theorem \ref{major} (3c), we know both $E_1(\F_q)$ and $E_2(\F_q)$ are cyclic, that is 
    $$
    E_1(\F_q)\cong Z/(q+1)\cong E_2(\F_q).
    $$
    When $m=2\sqrt{q}$, $\#E_1(\F_q)=\#E_2(\F_q)=p^r+1-2\sqrt{q}$. By Theorem \ref{major} (3a), we have 
    $$
    E_1(\F_q)\cong A\cong E_2(\F_q),
    $$
    where $A$ is given by 
    $$
    A=\Z/(\sqrt{q}-1)\times \Z/(\sqrt{q}-1).
    $$
    When $m=-2\sqrt{q}$, $\#E_1(\F_q)=\#E_2(\F_q)=p^r+1+2\sqrt{q}$. By (3)(a) in Theorem \ref{major}, we have 
    $$
    E_1(\F_q)\cong B\cong E_2(\F_q),
    $$
    where $B$ is given by 
    $$
    B=\Z/(\sqrt{q}+1)\times \Z/(\sqrt{q}+1).
    $$
    Therefore, we may conclude that when $q$ is an even power of $p$ and $p\equiv 3\mod 4$, two elliptic with $j$-invariant 1728 have the same number of $\F_q$ rational points if and only if their group of $\F_q$ rational points are isomorphic. 
\end{proof}

\vspace{0.4cm}
\subsection{Group Structure of Ordinary Curves}~{}\\
According to Theorem \ref{soro}, $E(\F_q)$ with $j$-invariant 1728 is ordinary if and only if $p\equiv 1\mod 4$. Then by Theorem \ref{thm: ord iso}, we have $\#E_1(K) = \#E_2(K)$ if and only if $E_1(K) \cong E_2(K)$.

\section{General Case: Not Depending on J-invariant}
In the preceding sections, we discussed when two elliptic curves with the same number of $\F_q$-rational points have isomorphic group structures of $\F_q$-rational points. Notice that all of the previous discussions require the two elliptic curves to have the same $j$-invariant. In this section, we provide a method to decide whether two elliptic curves $E_1,E_2$ defined over a finite field $\F_q$ have isomorphic groups of $F_q$ rational points, instead with the assumption that $\#E_1(F_q)=\#E_2(F_q)$. More details can be found in \cite{Heuberger}.

\indent Before we introduce this method, we first introduce some background knowledge about endomorphism rings and the Frobenius endomorphism. Let $\tau$ be the Frobenius endomorphism for an ordinary elliptic curve $E/\F_q$. It is well known that the endomorphism algebra of $E$, denoted $\Q(\tau)$, must be an imaginary quadratic field, and more specifically it is of the form $\Q(\sqrt{m})$ for some square free integer $m<0$. Thus, the endomorphism ring of $E$, being an order of the endomorphism algebra, can be written as 
$$
\text{End}(E)=\Z\oplus g\Z \delta,
$$
where $g$ is the conductor of this order, defined as the set $\{x \in \Q(\tau): x\mathcal{O}_K \subset \mathcal{O}\}$, and $\delta=\sqrt{m}$ if $m\equiv2,3\mod4$ and $\delta=\frac{1+\sqrt{m}}{2}$ if $m\equiv 1\mod 4$. Therefore, the Frobenius endomorphism $\tau$ can also be written as $\tau=a+b\delta$ for some integers $a,b$.

\begin{theorem}[\cite{Heuberger}, Theorem 2.4]
Let $E_1/\F_q$ and $E_2/\F_q$ be ordinary elliptic curves with $\#E_1(\F_q)=\#E_2(\F_q)$. Let $\tau=a+b\delta$ be their Frobenius endomorphism relative to $\F_q$, and let $End(E_1)=\mathcal{O}_1$ and $End(E_2)=\mathcal{O}_2$ be the orders in $\Q(\tau)$ of conductor $g_1$ and $g_2$ respectively. Let $\mathcal{P}:=\{p\text{ }\vert \text{ }v_p(g_1)\neq v_p(g_2) \text{ and } p \text{ is prime}\}$. For every $p\in \mathcal{P}$, define $s_p = \max \{v_p(g_1),v_p(g_2)\}$. \\
\indent Fix an integer $k\geq 1$, and let $\tau^k=a_k+b_k\delta$ for some $a_k,b_k\in \mathbb{Z}$. Then $$E_1(F_q)\cong E_2(F_q) \Leftrightarrow
v_p(a_k-1)\leq v_p(b_k)-s_p \text{ for all } p\in\mathcal{P},
$$
where $v_p(a)$ of an integer $a$ is the $p$-adic valuation of $a$.
\end{theorem}
Thus, one possible way to apply this theorem is by fixing $k$ to be, say, 1. After, carrying out the calculation of the above $p$-adic valuations, we will be able to see if $E_1(F_q)\cong E_1(F_q)$ given $\#E_1(F_q)=\#E_2(F_q)$, where $E_1$ and $E_2$ are of possibly different $j$-invariant.

\section*{Appendix: Data}\label{data}
\begin{center}
    \begin{tabular}{ |l|l|l|l|l|}
    \hline
    \multicolumn{5}{|c|}{$j = 0$, $r = 1$} \\
    \hline
    p & Order of EC's in Isog. Class & No. of EC & Group Structure(s) & Success \\
    \hline
    \multirow{1}{4em}{5} & 6 & 4 & $\znz{6}$ & Yes \\ 
    \hline
    \multirow{6}{4em}{7} & 24 & 1 & $\znzznz{2}{12}$ & \\
    & 9 & 1 & $\znzznz{3}{3}$ & Yes \\
    & 13 & 1 & $\znz{13}$ & \\
    & 3 & 1 & $\znz{3}$ & \\
    & 7 & 1 & $\znz{7}$ & \\
    & 4 & 1 & $\znzznz{2}{2}$ & \\
    \hline
    \multirow{1}{4em}{11} & 12 & 10 & $\znz{12}$ & Yes\\
    \hline
    \multirow{6}{4em}{13} & 12 & 2 & $\znzznz{2}{6}$ & \\
    & 19 & 2 & $\znz{19}$ & Yes \\
    & 9 & 2 & $\znzznz{3}{3}$ & \\
    & 21 & 2 & $\znz{21}$ & \\
    & 16 & 2 & $\znzznz{4}{4}$ & \\
    & 7 & 2 & $\znz{7}$ & \\
    \hline
    \multirow{1}{4em}{17} & 18 & 16 & $\znz{18}$ & Yes \\
    \hline
    \end{tabular}
    
    \vskip2em
    \begin{tabular}{ |l|l|l|l|l|}
    \hline
    \multicolumn{5}{|c|}{$j = 0$, $r = 2$} \\
    \hline
    p & Order of EC's in Isog. Class & No. of EC & Group Structure(s) & Success \\
    \hline
    \multirow{4}{4em}{5} & 31 & 8 & $\znz{31}$ &  \\ 
    & 21 & 8 & $\znz{21}$ & Yes \\ 
    & 16 & 4 & $\znzznz{4}{4}$ & \\ 
    & 36 & 4 & $\znzznz{6}{6}$ & \\
    \hline
    \multirow{6}{4em}{7} & 37 & 8 & $\znz{37}$ & \\
    & 39 & 8 & $\znz{39}$ & Yes \\
    & 52 & 8 & $\znzznz{2}{26}$ & \\
    & 63 & 8 & $\znzznz{3}{21}$ & \\
    & 61 & 8 & $\znz{61}$ & \\
    & 48 & 8 & $\znzznz{4}{12}$ & \\
    \hline
    \multirow{4}{4em}{11} & 133 & 40 & $\znz{133}$ & \\
    & 111 & 40 & $\znz{111}$ & Yes \\
    & 100 & 20 & $\znzznz{10}{10}$ & \\
    & 144 & 20 & $\znzznz{12}{12}$ & \\
    \hline
    \end{tabular}
    \vskip2em
    
    \begin{tabular}{ |l|l|l|l|l|}
    \hline
    \multicolumn{5}{|c|}{$j = 0$, $r = 3$} \\
    \hline
    p & Order of EC's in Isog. Class & No. of EC & Group Structure(s) & Success \\
    \hline
    \multirow{1}{4em}{5} & 126 & 124 & $\znz{126}$ & Yes \\
    \hline
    \multirow{6}{4em}{7} & 361 & 57 & $\znzznz{19}{19}$ & \\
    & 381 & 57 & $\znz{381}$ & Yes \\
    & 364 & 57 & $\znzznz{2}{182}$ & \\
    & 327 & 57 & $\znz{327}$ & \\
    & 307 & 57 & $\znz{307}$ & \\
    & 324 & 57 & $\znzznz{18}{18}$ & \\
    \hline
    \multirow{1}{4em}{11} & 1332 & 1330 & $\znz{1332}$ & Yes\\
    \hline
    \end{tabular}
    \vskip2em

    \begin{tabular}{ |l|l|l|l|l|}
    \hline
    \multicolumn{5}{|c|}{$j = 1728$, $r = 1$} \\
    \hline
    p & Order of EC's in Isog. Class & No. of EC & Group Structure(s) & Success \\
    \hline
    \multirow{4}{4em}{5} & 4 & 1 & $\znzznz{2}{2}$ & \\
    & 2 & 1 & $\znz{2}$ & Yes \\
    & 10 & 1 & $\znz{10}$ & \\
    & 8 & 1 & $\znzznz{2}{4}$ & \\
    \hline
    \multirow{1}{4em}{7} & 8 & 6 & $\znz{8}$, $\znzznz{2}{4}$ & No \\
    \hline
    \multirow{1}{4em}{11} & 12 & 10 & $\znzznz{2}{6}$, $\znz{12}$ & No\\
    \hline
    \multirow{4}{4em}{13} & 20 & 3 & $\znzznz{2}{10}$ & \\
    & 10 & 3 & $\znz{10}$ & Yes \\
    & 8 & 3 & $\znzznz{2}{4}$ & \\
    & 18 & 3 & $\znzznz{3}{6}$ & \\
    \hline
    \multirow{4}{4em}{17} & 16 & 4 & $\znzznz{4}{4}$ & \\
    & 20 & 4 & $\znzznz{2}{10}$ & Yes \\
    & 26 & 4 & $\znz{26}$ & \\
    & 10 & 4 & $\znz{10}$ & \\
    \hline
    \end{tabular}
    \vskip2em
    
    \begin{tabular}{ |l|l|l|l|l|}
    \hline
    \multicolumn{5}{|c|}{$j = 1728$, $r = 2$} \\
    \hline
    p & Order of EC's in Isog. Class & No. of EC & Group Structure(s) & Success \\
    \hline
    \multirow{4}{4em}{5} & 18 & 6 & $\znzznz{3}{6}$ &  \\ 
    & 20 & 6 & $\znzznz{2}{10}$ & Yes \\ 
    & 34 & 6 & $\znz{34}$ & \\ 
    & 32 & 6 & $\znzznz{4}{8}$ & \\
    \hline
    \multirow{3}{4em}{7} & 50 & 24 & $\znz{50}$ & \\
    & 36 & 12 & $\znzznz{6}{6}$ & Yes \\
    & 64 & 12 & $\znzznz{8}{8}$ & \\
    \hline
    \multirow{3}{4em}{11} & 122 & 60 & $\znz{122}$ & \\
    & 100 & 30 & $\znzznz{10}{10}$ & Yes \\
    & 144 & 30 & $\znzznz{12}{12}$ & \\
    \hline
    \end{tabular}
    
    \vskip2em
    \begin{tabular}{ |l|l|l|l|l|}
    \hline
    \multicolumn{5}{|c|}{$j = 1728$, $r = 3$} \\
    \hline
    p & Order of EC's in Isog. Class & No. of EC & Group Structure(s) & Success \\
    \hline
    \multirow{4}{4em}{5} & 130 & 31 & $\znz{130}$ &  \\ 
    & 104 & 31 & $\znzznz{2}{52}$ & Yes \\ 
    & 122 & 31 & $\znz{122}$ & \\ 
    & 148 & 31 & $\znzznz{2}{74}$ & \\
    \hline
    \multirow{1}{4em}{7} & 344 & 342 & $\znzznz{2}{172}$, $\znz{344}$ & No \\
    \hline
    \multirow{1}{4em}{11} & 1332 & 1330 & $\znzznz{2}{666}$, $\znz{1332}$ &  No\\
    \hline
    \end{tabular}
\end{center}

%\begin{remark}
%R\"{u}ck has classified in his paper \textit{A note on elliptic curves over finite fields} the possible pairs $(n_1, n_2)$ that can occur.
%\end{remark}


\begin{thebibliography}{1}


\bibitem{Kim}
J. Bahr, Y. H. Kim, E. Neyman and G. Taylor, \emph{On Orders of Elliptic Curves over Finite Fields}, {\bf Rose-Hulman Undergraduate Mathematics Journal}, Vol. 19:1 (2018), Article 2. 



\bibitem{Cox}
D. A. Cox, \emph{Primes of the form $x^2+ny^2$}, {\bf A Wiley-Interscience Publication}, John Wiley \& Sons Inc., (1989).

\bibitem{Davenport}
H. Davenport, \emph{Multiplicative Number Theory}, Vol.~74 of {Graduate Texts in Mathematics}, {\bf Springer-Verlag}, 2nd~ed., (1980).



\bibitem{DiamondShurman}
F. Diamond and J.Shurman, \emph{A First Course in Modular Forms}, Vol.~228 of {Graduate Texts in Mathematics}, {\bf Springer-Verlag}, 1st~ed., (2005).

\bibitem{Galbraith}
S.D. Galbraith, \emph{Constructing Isogenies Between Elliptic Curves over Finite Fields}, {\bf Journal of Computational Mathematics}, Vol. 2 (1999), 118--138.


\bibitem{Heuberger}
C. Heuberger and M. Mazzoli, \emph{Elliptic curves with isomorphic groups of points over finite field extensions}, {\bf Journal of Number Theory}, Vol.181 (2017), 89--98. 

\bibitem{IR}
K. Ireland and M. Rosen, {\em A Classical Introduction to Modern Number Theory}, Vol.~84 of{Graduate Texts in Mathematics}, {\bf Springer-Verlag}, 2nd edition (1998).


\bibitem{Lenstra}
H.W. Lenstra Jr., \emph{Factoring Integers with Elliptic Curves}, {\bf The Annals of Mathematics: Second Series}, Vol. 126:3 (1987), 649--673.



\bibitem{Mazur}
B. Mazur, \emph{Rational Points of Abelian Varieties with Values in Towers of Number Fields}, {\bf Invent. Mathematics} Vol. 18 (1972), 183--266.



 \bibitem{Munuera}
C. Munuera and Juan G. Tena, \emph {An Algorithm to Compute the Number of Points on Elliptic Curves of j-invariant 0 or 1728 Over a Finite Field}, {\bf Rendiconti del Circolo Matematico di Palermo} Vol. 42:1 (1993) 106--116.




\bibitem{Schoof}
R. Schoof, \emph{Counting Points on Elliptic Curves over Finite Fields}, {\bf J. Theory Nombres Bordeaux} Vol. 7 (1995), 219--254.



\bibitem{Silverman}
J.H. Silverman, \emph {The Arithmetic of Elliptic Curves}, Vol.~106 of {\bf Graduate Texts in Mathematics}, {\bf Springer-Verlag}, 1st~ed., (1986).

\bibitem{Tate}
J. Tate, \emph{Endomoprhisms of abelian varieties over finite fields}, {\bf Inventiones Mathematica}, Vol.2 (1966), 134--144. 


\bibitem{Velu}
J. Velu, \emph{Isogenies entre courbes elliptiques}, {\bf C. R. Acad. Sci. Paris Ser. I}, Vol. 273 (1971), 238--241.

\bibitem{Vladut}
S.G. Vl\u{a}du\c{t}, \emph{On the Cyclicity of Elliptic Curves overFinite Field Extensions}, {\bf Finite Fields and Their Applications}, Vol. 5:4 (1999), 354--363.


\bibitem{Wl}
L.C. Washington, \emph{Number Theory: Elliptic Curves and Cryptography}, {\bf Discrete Mathematics and Its Applications}, Chapman \& Hall/CRC, 2nd~ed., (2008).

\bibitem{Waterhouse}
W. C. Waterhouse, \emph{Abelian varieties over finite fields}, {\bf Ann. Sci. Ecole Norm. Sup} Vol. 4:2 (1969), 521--560.

\bibitem{Wittmann}
C. Wittmann, \emph{Group Structure of Elliptic Curves over Finite Fields}, {\bf Journal of Number Theory} Vol.88 (2001), 335--344.

\end{thebibliography}
\end{document}